\documentclass[11pt,reqno]{amsart}

\usepackage{amsmath,amsthm,amssymb,mathrsfs,mathtools,stmaryrd,color}
\usepackage{empheq}

\usepackage[top=1in, bottom=1in, left=1in, right=1in]{geometry}

\usepackage{float}
\usepackage{graphicx}

\usepackage[utf8]{inputenc}
\usepackage[T1]{fontenc}

\usepackage{enumitem}
\setlist[enumerate]{itemsep=2pt,parsep=2pt,before={\parskip=2pt}}

\usepackage[colorlinks=true,hyperindex, linkcolor=red!60, pagebackref=false, citecolor=cyan, pdfpagelabels]{hyperref}
\usepackage[capitalize]{cleveref}
\crefname{equation}{}{}
\crefformat{enumi}{#2\textup{\textcolor{black}{(}#1\textcolor{black}{)}}#3}

\usepackage{url}
\usepackage{breakurl}

\usepackage{tikz}
\usetikzlibrary{calc}
\usetikzlibrary{math}

\usepackage{thmtools}
\usepackage{thm-restate}

\newtheorem{theorem}{Theorem}[section]
\newtheorem*{theorem*}{Theorem}
\newtheorem*{definition*}{Definition}
\newtheorem{problem}[theorem]{Problem}

\newtheorem{lemma}[theorem]{Lemma}

\theoremstyle{definition}
\newtheorem{definition}[theorem]{Definition}

\newcommand{\R}{\ensuremath{\mathbb{R}}}
\newcommand{\Z}{\ensuremath{\mathbb{Z}}}
\newcommand{\N}{\ensuremath{\mathbb{N}}}
\newcommand{\eps}{\varepsilon}
\newcommand{\norm}[1]{\left\|#1\right\|}

\newcommand{\abs}[1]{\left|#1\right|}
\newcommand{\ind}[1]{\mathbf{1}_{#1}}

\newcommand{\base}{\boldsymbol{b}}
\newcommand{\digits}[1]{\overline{#1}^{\base}}

\newcommand{\poly}{\mathbb{F}_q[t]}
\newcommand{\F}{\mathcal{F}}
\newcommand{\cst}{c}
\newcommand{\Cst}{C_0}
\newcommand{\Pintro}{\mathcal{P}}
\newcommand{\irred}[1]{\mathcal{P}_{#1}}
\renewcommand{\Pr}[1]{\mathbb{P}\left(#1\right)}
\newcommand{\Prbis}{\mathbb{P}}
\newcommand{\E}{\mathcal{E}}
\newcommand{\EE}[1]{\mathbb{E}\left[#1\right]}
\newcommand{\EEbis}{\mathbb{E}}
\newcommand{\optionalsign}{-}
\newcommand{\fmod}[2]{\left[#1 \ \mathrm{mod}\ #2 \right]}
\newcommand{\pp}{p}
\newcommand{\triples}{\mathcal{T}}
\newcommand{\tupples}{\mathcal{Q}}
\renewcommand{\leq}{\leqslant}
\renewcommand{\geq}{\geqslant}

\makeatletter
\@namedef{subjclassname@2020}{%
  \textup{2020} Mathematics Subject Classification}
\makeatother

\begin{document}

\title[A solution to the Erd\H{o}s-S\'{a}rk\"{o}zy-S\'{o}s problem on asymptotic Sidon bases of order 3]{A solution to the Erd\H{o}s-S\'{a}rk\"{o}zy-S\'{o}s problem \\on asymptotic Sidon bases of order 3}

\author{C\'edric Pilatte}
\address{Mathematical Institute, University of Oxford.}
\email{cedric.pilatte@maths.ox.ac.uk}



\maketitle

\begin{abstract}
A set $S\subset \N$ is a \emph{Sidon set} if all pairwise sums $s_1+s_2$ (for $s_1, s_2\in S$, $s_1\leq s_2$) are distinct. A set $S\subset \N$ is an \emph{asymptotic basis of order 3} if every sufficiently large integer $n$ can be written as the sum of three elements of $S$. In 1993, Erd\H{o}s, S\'{a}rk\"{o}zy and S\'{o}s asked whether there exists a set $S$ with both properties. We answer this question in the affirmative. Our proof relies on a deep result of Sawin on the $\mathbb{F}_q[t]$-analogue of Montgomery's conjecture for convolutions of the von Mangoldt function.
\end{abstract}

\section{Introduction}

A set $S$ of natural numbers is called a \emph{Sidon set}, or a \emph{Sidon sequence}, if all the sums $s_1+s_2$, for $s_1, s_2\in S$, $s_1\leq s_2$, are distinct. Sidon sequences are named after Simon Sidon, who in 1932 asked Erd\H{o}s about the possible growth rate of such sequences. Write 
\begin{equation*}
	S(x) := \abs{ \{n\leq x: n\in S\} }.
\end{equation*}
Erd\H{o}s observed that the greedy algorithm generates a Sidon sequence $S$ with $S(x) \gg x^{1/3}$. He also conjectured that, for every $\eps>0$, there is a Sidon sequence $S$ with $S(x) \gg x^{1/2-\eps}$ \cite{erdosstory}. In some sense, this would be best possible, as Erd\H{o}s showed that any Sidon sequence $S$ satisfies
\begin{equation*}
	S(x) \ll {x^{1/2}}{(\log x)^{-1/2}}
\end{equation*}
for infinitely many $x$ \cite[Theorem~8]{rothhalberstam}. The lower bound was improved in 1981 by Ajtai, Koml{\'o}s and Szemer{\'e}di \cite{szemeredi}, who proved the existence of a Sidon sequence $S$ with 
\begin{equation*}
	S(x) \gg x^{1/3} \log x^{1/3}
\end{equation*}
using graph-theoretic tools. In a 1998 landmark paper, Ruzsa \cite{ruzsa} used the fact that the primes form a ``multiplicative Sidon set'' to construct a Sidon sequence $S$ with
\begin{equation*}
	S(x) \gg x^{\sqrt{2}-1-o(1)}.
\end{equation*}
This is still the best-known lower bound: improving the exponent $\sqrt{2}-1 \approx 0.4142\ldots$ would be a major achievement.

It is worth mentioning that much more is known about \emph{finite} Sidon sets. In particular, the maximal size of a Sidon subset of $\{1,2,\ldots, n\}$ is $n^{1/2}+O(n^{1/4})$ (see Halberstam and Roth \cite[Chapter~I,~Section~3]{rothhalberstam} for detailed references).

Sidon sequences have become classical objects of interest in arithmetic combinatorics. It is natural to study the properties of the sumset $S+S := \{s_1+s_2:s_1, s_2\in S\}$, or more generally of the $k$th iterated sumset ${kS := \underbrace{S+S+\ldots+S}_{k \text{ times}}}$, when $S$ is a Sidon sequence.

A subset $A\subset \N$ is called an \emph{asymptotic basis of order $k$} if $\N \setminus kA$ is finite. In other words, $A$ is an asymptotic basis of order $k$ if every sufficiently large integer can be written as the sum of $k$ elements of $A$. Many well-known problems can be restated as follows: for a given set $A$, what is the smallest $k$ such that $A$ is an asymptotic basis of order $k$? If $A$ is the set of primes, this is essentially Goldbach's conjecture, while if $A$ is the set of perfect $d$th powers, this is a variant of Waring's problem.

In \cite[Problem~14]{erdos1} and \cite[Problem~8]{erdos2}, Erd\H{o}s, S\'{a}rk\"{o}zy and S\'{o}s asked the following question.
\begin{problem}
\label{problem}
Does there exist a Sidon sequence $S$ which is an asymptotic basis of order 3?
\end{problem}
\vspace{-5pt}
This problem also appears in Erd\H{o}s \cite[p.212]{erdos3} and S\'{a}rk\"{o}zy \cite[Problem~32]{sarkozy}\footnote{There is a typo in \cite[Problem~32]{sarkozy}: ``2'' should be replaced with ``3''.}. In his paper \cite{erdos3}, Erd\H{o}s describes \cref{problem} as ``an old problem of Nathanson and myself''. 

It is easy to see that no Sidon sequence can be an asymptotic basis of order $2$. In fact, Erd\H{o}s, S\'{a}rk\"{o}zy and S\'{o}s \cite{erdos2bis} proved that for any Sidon set $S\subset \{1, 2, \ldots, n\}$, the sumset $S+S$ cannot contain $\geq Cn^{1/2}$ consecutive integers, where $C$ is an absolute constant. Thus, the constant ``$3$'' in \cref{problem} can certainly not be improved.

In this article, we settle \cref{problem} by proving the existence of a Sidon sequence $S$ which is also an asymptotic basis of order $3$. 

A number of partial results have been previously obtained in the direction of \cref{problem}.
\begin{center}
    \begin{tabular}{ c|c } 
    	There is an asymptotic Sidon basis of order & Reference\\
         \hline
         7 & Deshoulliers and Plagne \cite{sidon7} \\ 
         5 & Kiss \cite{sidon5} \\ 
         4 & Kiss, Rozgonyi and S{\'a}ndor \cite{sidon4} \\ 
         $3+\eps$ & Cilleruelo \cite[Theorem~1.3]{sidon3plus}
    \end{tabular}
\end{center}
To make sense of the last row, we recall that a set $S\subset \N$ is an asymptotic basis of order $3+\eps$ if, for every $\eps>0$, every sufficiently large integer $n$ can be written as the sum of four elements of $S$, with one of them $\leq n^{\eps}$. Our solution to \cref{problem} also generalises another result of Cilleruelo \cite[Theorem~1.2]{sidon3plus}, which states that there is an asymptotic basis $S$ of order $3$ such that every integer $n$ can be represented in at most two different ways as a sum of two elements of $S$. Finally, Kiss and S{\'a}ndor \cite{dense} showed the existence of a Sidon sequence $S\subset \N$ whose $3$-fold sumset $S+S+S$ has lower asymptotic density $>0.0064$.\footnote{The constant $0.064$ in their paper \cite{dense} is a typo.} 

The results in \cite{sidon3plus,sidon5,sidon4,dense} quoted above all use some variant of the probabilistic method. The starting point of all these approaches is to consider a random subset $S$ of $\N$ where each $n\in \N$ is chosen to be in~$S$, independently, with some probability $p(n)$ -- a typical choice is $p(n) = n^{-\gamma}$ for some $\gamma>0$. In order for $S$ to resemble a Sidon sequence (i.e.~to be able to use the alteration method), it is necessary to have something roughly like $p(n) \ll n^{-2/3}$. In \cite{sidon3plus,sidon5,sidon4}, the authors choose $p(n) = n^{-\gamma}$ for some $\gamma<2/3$, while for \cite{dense} they choose $p(n) = c n^{-2/3}$ for some optimised constant~$c>0$.
 However, for such choices of $p(n)$, the set $S$ will almost surely \emph{not} be an asymptotic basis of order $3$ (this follows from \cite[Satz~B]{goguel}). Hence, it seems that purely probabilistic approaches are of little use for addressing \cref{problem}.

In order to obtain an asymptotic basis of order $3$, it is more promising to start from the Sidon sequence of Ruzsa mentioned above \cite{ruzsa}, as it is much denser than the Sidon sets obtained by the probabilistic method. 

The underlying idea behind Ruzsa's construction is to consider an infinite set of primes $\Pintro$, and for each $p\in \Pintro$, to define a natural number $n_p$ that ``behaves like a logarithm of $p$'', so that an equality $n_{p_1} + n_{p_2} = n_{p_3} + n_{p_4}$ only holds if $p_1p_2 = p_3p_4$. Since the primes form a multiplicative Sidon set, we get $\{p_1, p_2\} = \{p_3, p_4\}$, so the set $S := \{n_p : p\in \Pintro\}$ has the Sidon property. In his original paper, Ruzsa defined $n_p$ in terms of the binary expansion of the real number $\log_b{p}$ (for some~$b\in \R^{>1}$) by concatenating certain blocks of digits of $\log_b{p}$ to obtain a natural number. 

Cilleruelo \cite{cilleruelo} proposed a neat variant of Ruzsa's construction, where the real logarithm is replaced by a family of discrete logarithms. Fix an increasing sequence $\ell_1, \ell_2, \ldots$ of primes $\ell_i\not\in \Pintro$ and choose for each $i$ a primitive root $\omega_i$ modulo $\ell_i$. Write $\log_{\ell_i, \omega_i}(p)$ for the unique $e\in \{0, 1, \ldots, \ell_i-2\}$ such that $\omega_i^e \equiv p  \pmod {\ell_i}$. Cilleruelo defined $n_p$ by encoding many of these discrete logarithms $\log_{\ell_i, \omega_i}(p)$ into a single natural number, using a suitable base expansion. Now, an equality of the form $n_{p_1} + n_{p_2} = n_{p_3} + n_{p_4}$ implies that
\begin{equation*}
	\log_{\ell_i, \omega_i}(p_1) + \log_{\ell_i, \omega_i}(p_2) = \log_{\ell_i, \omega_i}(p_3) + \log_{\ell_i, \omega_i}(p_4)
\end{equation*}
for many values of $i$, and thus $p_1p_2 \equiv p_3p_4 \pmod {\ell_i}$ for many $i$. With the appropriate quantification, this is only possible if $p_1p_2 = p_3p_4$, which forces $\{p_1, p_2\} = \{p_3, p_4\}$ as before. 

Neither Ruzsa's nor Cilleruelo's construction constitutes an asymptotic basis of order $3$. There is an obvious obstruction, which is intrinsically tied to the way the information about $p$ is encoded in $n_p$. For the readers familiar with Ruzsa's construction, the reason is due to the presence of zeros between the blocks of digits of $\log_b{p}$ within $n_p$. These zeros cannot be removed, for this would destroy the Sidon property of $S$. Cilleruelo's construction suffers from a similar drawback. 

Our point of departure is Cilleruelo's construction. To avoid the issue raised in the previous paragraph, we replace the zeros in the encoding of $n_p$ with random numbers chosen from a carefully designed set $A$ (see \cref{lem:specialset} and \cref{def:ourSequence}). The remaining task is to prove that our modified Cilleruelo construction produces an asymptotic basis of order $3$. Eventually, this reduces to information about the distribution of products of three primes in arithmetic progressions. 

However, the kind of information we would need is not remotely within reach of our current knowledge of prime numbers. We would need something along the lines of Montgomery's conjecture, a far-reaching generalisation of the Generalised Riemann Hypothesis. Montgomery's conjecture~\cite[Eq.\,(17.5), p.419]{iwanieckowalski} states that, for every $\eps>0$,
\begin{equation}
\label{eq:montgomery}
	\psi(x;q,a) := \sum_{\substack{n\leq x\\ n\equiv a\,(\mathrm{mod}~q)}} \Lambda(n) = \frac{x}{\varphi(q)} + O_{\eps}\left(q^{-1/2}x^{1/2+\eps}\right)
\end{equation}
uniformly in $x, q\geq 1$ and $(a, q)=1$. This implies the asymptotic formula $\psi(x;q,a) \sim x/\varphi(q)$ uniformly in the residue class and the modulus, provided that $q\leq x^{1-\eps}$. We would need a similar (suitably normalised) statement, with the triple convolution $\Lambda *\Lambda *\Lambda$ in place of the von Mangoldt function~$\Lambda$.

Cilleruelo \cite[Section~4]{cilleruelo} observed that his construction works equally well over $\poly$. Let $\F$ be a set of irreducible monic polynomials in $\poly$. Fix a sequence $(g_i)$ of irreducible monic polynomials not in $\F$, and for each $g_i$, choose a generator $\omega_i$ of $(\poly/(g_i))^{\times}$. For $f\in \F$, write $\log_{g_i, \omega_i}(f)$ for the unique integer $e\in \{0, 1, \ldots, q^{\deg g_i}-1\}$ such that $\omega_i^{e} \equiv f\pmod{g_i}$. As before, one can define an integer $n_f$ in terms of many such integers $\log_{g_i, \omega_i}(f)$, so that $S := \{n_f : f\in \F\}$ is a Sidon sequence. It is important to remember that $S$ is a Sidon sequence of \emph{integers}, even if $\F$ is a subset of $\poly$.

The final and crucial ingredient to our proof is the remarkable work of Sawin \cite{sawin}, who proved Montgomery-type results for ``factorisation functions'' in $\poly$, a class of arithmetic functions that encompasses (the $\poly$-version of) the von Mangoldt function and its convolutions. For example, \cite[Theorem~1.2]{sawin} implies that, for all $\eps>0$, there is some $q_0(\eps)$ such that, for all prime powers $q\geq q_0(\eps)$, the analogue of Montgomery's conjecture \cref{eq:montgomery} holds over $\poly$ (for squarefree moduli), i.e.
\begin{equation}
\label{eq:montgomerysawin}
	\sum_{\substack{f\in \poly\\ f\text{ monic of degree }n\\ f\equiv a\,(\mathrm{mod}~g)}} \Lambda(f) = \frac{q^n}{\varphi(g)} + O\left((q^{\deg g})^{-1/2}(q^n)^{1/2+\eps}\right)
\end{equation}
uniformly in the modulus $g$ (squarefree monic polynomial) and in the residue class $a$, with ${(a,g)=1}$. Sawin's work relies on deep algebraic geometry methods, including sheaf cohomology, the characteristic cycle, vanishing cycles theory and perverse sheaves. We will use a bound similar to \cref{eq:montgomerysawin} for the triple convolution $\Lambda *\Lambda *\Lambda$ (see \cref{lem:sawin}). With this powerful result, we are able to prove that our modified Cilleruelo sequence is an asymptotic basis of order $3$, thereby solving \cref{problem}.

To conclude this introduction, we mention an application of our work to the study of $B_h[1]$ sequences. A set $S\subset \N$ is said to be $B_h[g]$ set if every positive integer can be written as the sum of $h$ terms from $S$ at most $g$ different ways. For any $h\geq 2$, with the probabilistic method, Kiss and S{\'a}ndor successively showed the existence of a $B_h[1]$ set which is an asymptotic basis of order $2h+1$ \cite{kiss2hplus1}, then $2h$ \cite{kiss2h}. A simple modification of our proof should establish the existence of a $B_h[1]$ set which is an asymptotic basis of order $2h-1$, for $h\geq 2$, thus answering a question of Kiss and S{\'a}ndor~\cite{kiss2hplus1,kiss2h}.

\section*{Acknowledgements}
The author is supported by the Oxford Mathematical Institute and a Saven European Scholarship. I would like to thank my advisors, Ben Green and James Maynard, for their expert guidance and continuing encouragement. I am very grateful to Oliver Riordan for pointing out an error in the appendix of an earlier version of this paper, and for suggesting the method used here to fix it.

\section{Notation}

The sumset of two sets $A,B\subset \Z$ is $A + B := \{a+b : a\in A, \, b\in B\}$. We write $f \ll g$ or $f = O(g)$ if $\abs{f} \leq C g$ for some absolute constant $C>0$. If instead $C$ depends on a parameter $\theta$, we write $f \ll_{\theta} g$ or $f = O_{\theta}(g)$.

For $b\geq 1$ and $a\in \Z$, we write $\fmod{a}{b}$ for the unique integer $n$ satisfying $0\leq n < b$ and ${n\equiv a \pmod b}$.

\begin{definition}[Generalised base]
Let $\base = (b_1, b_2, \ldots)$ be an infinite sequence of integers $\geq 2$. For any $n\geq 1$ and any $x_1, \ldots, x_n \in \Z$, we write 
\begin{equation*}
	\digits{x_n\, \ldots\,  x_1} := x_1 + x_2b_1 + x_3b_1b_2 + \ldots + x_n b_1b_2\cdots b_{n-1}.
\end{equation*}
\end{definition}

Any $m\in \Z^{\geq1}$ can be uniquely represented as $m = \digits{x_n\, \ldots\,  x_1}$ for some $0\leq x_i<b_i$ with $x_n\neq 0$. Recall however that the notation $\digits{x_n\, \ldots\,  x_1}$ is defined for arbitrary integers $x_i$: we will not always require $0\leq x_i<b_i$.


\section{Construction}

The following lemma is a slight strengthening of the statement that, for any sufficiently large $p$, there exists a set $A\subset \Z/p\Z$ such that $A$ and $A+A$ are disjoint, and moreover $A+A+A = \Z/p\Z$. 

\begin{lemma}
\label{lem:specialset}
For every sufficiently large prime $\pp$, there is a set $A \subset \{1, 2, \ldots, \lfloor \pp/2 \rfloor - 1\}$ such that
\begin{enumerate}[label=(\roman*), ref=\roman*]
	\item \label{item:disjoint} the sets $A$ and $A+A+\{0, 1\}$ are disjoint;
	\item \label{item:covers} $A+A+A$ contains $p+2$ consecutive integers.
\end{enumerate}
\end{lemma}

It turns out that we will only need to use the existence of a single pair $(p, A)$ with these properties. \Cref{lem:specialset} can be shown by the alteration method in probabilistic combinatorics. We provide a detailed proof in \cref{sec:appendix}.

For the remainder of this paper, we fix a pair $(\pp, A)$ satisfying the conclusion of \cref{lem:specialset}. In particular, $\pp$ should be thought of as an absolute constant. Let $\cst = 0.35$, the point is that
\begin{equation*}
\frac13 < \cst < \frac{3-\sqrt{5}}{2}.
\end{equation*}
Let $\Cst > 100\pp$ be a large absolute constant that will be chosen later.

\begin{definition}[$\irred{d}$, $g_i$, $\omega_i$, $\F_k$, $\F$]
\label{def:poly}
Let $q\geq \Cst$ be a prime (or a prime power). Let $\irred{d}$ be the set of irreducible monic polynomials $f\in \poly$ of degree $d$. Recall the standard formula of Gauss
\begin{equation}
\label{eq:gauss}
	|\irred{d}| = \frac{1}{d}\sum_{e \mid d} \mu \left(\frac{d}{e}\right) q^{e} = \frac{q^d}{d} + O(q^{d/2}).
\end{equation}

For $i\geq 1$, let $g_i$ be an arbitrary element of $\irred{2i-1}$, and fix an arbitrary generator $\omega_i$ of $(\poly/(g_i))^{\times}$. 

For $k\geq \Cst$, let 
\begin{equation*}
	\F_k := \bigcup_{\substack{\cst k^2\leq 2i < \cst (k+1)^2}} \irred{2i}.
\end{equation*}
Let $\F := \bigcup_{k\geq \Cst} \F_k$.
\end{definition}

We will work in the generalised base $\base = (b_1, b_2, \ldots)$ where
\begin{equation}
\label{eq:defbase}
	b_i = \begin{cases}
		q^{i}-1 & i\text{ odd}\\
		\pp & i \text{ even.}
	\end{cases}
\end{equation}

\begin{definition}[$e_i$, $r_i$, $s$, $n_f$, $S$]
\label{def:ourSequence}
Let $k\geq \Cst$ and let $f\in \F_k$. We associate to $f$ a positive integer $n_f$ as follows. 
\begin{itemize}
	\item For $1\leq i\leq k$, let $e_i(f)$ be the unique integer such that $0\leq e_i(f)< b_{2i-1} = q^{2 i-1}-1$ and 
\begin{equation*}
	\omega_i^{e_i(f)} \equiv f \pmod {g_i}.
\end{equation*}
	\item Let $r_1(f), \ldots, r_k(f)$ be a sequence of i.i.d.~random variables, each uniformly distributed on~$A$ (in particular, $1\leq r_i(f) \leq \pp/2-1$). 
	\item Let $s(f)$ be an integer chosen uniformly at random in the set $\{1, 2, \ldots, q^{3k}\}$.
\end{itemize}
It is understood that the family of all random variables $r_i(f)$ and $s(f)$ (over all choices of $f$ and $i$) is independent.

We define $n_f$ to be the integer 
\begin{equation*}
	n_f := \digits{s\, r_k\, e_k\, \ldots\, r_2\, e_2\, r_1\, e_1}
\end{equation*}
omitting the dependence of each digit on $f$ for conciseness. To be explicit, 
\begin{equation*}
	n_f = e_1(f) + r_1(f) b_1 + e_2(f) b_1b_2 + r_2(f) b_1 b_2 b_3 + \ldots + r_k(f) b_1b_2\cdots b_{2k-1} + s(f) b_1b_2\cdots b_{2k}, 
\end{equation*}
with $b_i$ as in \cref{eq:defbase}.\footnote{Note that $s(f)$ may exceed $b_{2k+1}$, whereas for $1\leq i\leq 2k$ the $i$th digit ($e_{(i+1)/2}$ or $r_{i/2}$) is always between $0$ and $b_i-1$.}

Finally, we define $S = \{n_f : f\in \F\}$.
\end{definition}

We will show that, with probability $1$, the elements of $S$ form a Sidon sequence and an asymptotic basis of order $3$.
\medbreak

\begin{lemma}
\label{lem:preliminary}
\begin{enumerate}[label=(\roman*), ref=\roman*]
	\item \label{item:sizenf} Let $f\in \F_k$. Then $n_f = q^{k^2+O(k)}$.
	\item \label{item:uniqueness} Let $f_1, f_2\in \F$ be such that $n_{f_1} = n_{f_2}$. Then $f_1 = f_2$.
\end{enumerate}
\end{lemma}
\begin{proof}
\begin{enumerate}[label=(\roman*), ref=\roman*]
	\item On the one hand, 
	\begin{equation*}
		n_f = \digits{s\, r_k\, e_k\, \ldots\, r_2\, e_2\, r_1\, e_1} \geq \digits{1\, \underbrace{0\, 0\,  \ldots\, 0\, 0}_{2k\text{ zeros}}} = \prod_{i=1}^k \pp(q^{2i-1}-1) > \prod_{i=1}^k q^{2i-1} = q^{k^2}.
	\end{equation*}
	On the other hand,
	\begin{equation*}
		n_f = \digits{s\, r_k\, e_k\, \ldots\, r_2\, e_2\, r_1\, e_1} \leq \digits{(q^{3k}+1)\, \underbrace{0\, 0\,  \ldots\, 0\, 0}_{2k\text{ zeros}}} = (q^{3k}+1) \prod_{i=1}^k \pp(q^{2i-1}-1) = q^{k^2+O(k)},
	\end{equation*}
	as $\pp \leq q$. This shows that $n_f = q^{k^2+O(k)}$.
	\item Let $n = n_{f_1} = n_{f_2}$. Suppose that $f_1\in \F_{k_1}$ and $f_2\in \F_{k_2}$. By \cref{item:sizenf}, we have
	\begin{equation*}
		q^{k_1^2 + O(k_1)} = n = q^{k_2^2 + O(k_2)},
	\end{equation*}
	from which we infer that ${k_1 = k_2+O(1)}$. Let $k = \min(k_1, k_2)$. Since $n_{f_1} = n_{f_2} = n$, we also have
	\begin{equation*}
		\digits{r_k(f_1)\, e_k(f_1)\, \ldots\, r_1(f_1)\, e_1(f_1)}  =  \fmod{n}{b_1b_2\cdots b_{2k}} = \digits{r_k(f_2)\, e_k(f_2)\, \ldots\, r_1(f_2)\, e_1(f_2)}.
	\end{equation*}
	The left and right-hand side are two base-$\base$ representations of the same natural number, and since the $i$th digit is between $0$ and $b_i-1$ in both cases, we see that ${e_i(f_1) = e_i(f_2)}$ (and ${r_i(f_1) = r_i(f_2)}$) for all ${1\leq i\leq k}$. By definition of $e_i$, this implies that ${f_1 \equiv f_2\pmod {g_i}}$. By the Chinese remainder theorem, we deduce that 
	\begin{equation}
	\label{eq:congruence1}
		{f_1 \equiv f_2\pmod {g_1\cdots g_k}}.
	\end{equation}
	However, $\deg(f_1-f_2) \leq \cst k^2 + O(k)$ by definition of $\F_k$, whereas $\deg(g_1\cdots g_k) = k^2$. Therefore, $\deg(f_1-f_2) < \deg(g_1\cdots g_k)$ if $\Cst$ is sufficiently large, and \cref{eq:congruence1} can only hold if $f_1 = f_2$. \qedhere
\end{enumerate}
\end{proof}

\section{Sidon property}

\begin{lemma}
\label{lem:sidon}
Let $f_1, f_2, f_3, f_4\in \F$ be such that $n_{f_1} + n_{f_2} = n_{f_3} + n_{f_4}$. Then $\{f_1, f_2\} = \{f_3, f_4\}$.
\end{lemma}
\begin{proof}
Without loss of generality, suppose that $f_i\in \F_{k_i}$ for $i\in \{1,2,3,4\}$, where $k_1\geq k_2$ and ${k_3\geq k_4}$. By part \cref{item:sizenf} of \cref{lem:preliminary}, we have $q^{k_1^2+O(k_1)} = n_{f_1} + n_{f_2} = n_{f_3} + n_{f_4} = q^{k_3^2 + O(k_3)}$, and thus $k_1 = k_3 + O(1)$. 

We claim that $k_2 = k_4 + O(1)$. Write the integer $n := n_{f_1} + n_{f_2} = n_{f_3} + n_{f_4}$ in base $\base$, as
\begin{equation*}
	n = \digits{y_l\, x_l\, \cdots \, y_2\, x_2\, y_1\, x_1},
\end{equation*}
for some $0\leq x_i < q^{2i-1}-1$ and $0\leq y_1 < \pp$, where $x_l$ and $y_l$ are not both zero. Let us align the base-$\base$ digits of $n_{f_1}$, $n_{f_2}$ and $n$:
\begin{equation*}
	\begin{matrix}
		n_{f_1}: & s(f_1) & r_{k_1}(f_1) & e_{k_1}(f_1) & \cdots & r_{k_2+1}(f_1) & e_{k_2+1}(f_1) & r_{k_2}(f_1) & e_{k_2}(f_1) & \cdots & r_{1}(f_1) & e_{1}(f_1)\\
		n_{f_2}: & & & & & & s(f_2) & r_{k_2}(f_2) & e_{k_2}(f_2) & \cdots & r_{1}(f_2) & e_{1}(f_2)\\
		\hline
		n:& \cdots& y_{k_1} & x_{k_1} & \cdots &y_{k_2+1} & x_{k_2+1} & y_{k_2} & x_{k_2} & \cdots & y_1 & x_1.
	\end{matrix}	
\end{equation*}
Since $n_{f_1} + n_{f_2} = n$ and $r_i(f_j) \leq \pp/2-1$ for all $i$ and $j$, we can easily express the digits $x_i$ and $y_i$ in terms of the digits of $n_{f_1}$ and $n_{f_2}$. The computation gets slightly more complicated when $i$ is close to $k_2$ (or larger than $k_1$) because $s(f_2)$ can be as large as $q^{3k_2}$, which exceeds $b_{2k_2+1} = q^{2k_2+1}-1$.

On the one hand, for $1\leq i\leq k_2$, we have
\begin{itemize}
	\item $x_i = \fmod{e_i(f_1) + e_i(f_2)}{q^{2i-1}-1}$, and
	\item $y_i = r_i(f_1) + r_i(f_2) + \ind{e_i(f_1) + e_i(f_2) \geq q^{2i-1}-1}$.
\end{itemize}
On the other hand, for $k_2+3\leq i\leq k_1$ we have\footnote{Note that $s(f_2) \leq q^{3k_2} < b_{2k_2+1}b_{2k_2+2}b_{2k_2+3}$, which is why these formulas are valid for $k_2+3 \leq i\leq k_1$.}
\begin{itemize}
	\item $x_i = e_i(f_1)$, and
	\item $y_i = r_i(f_1)$.
\end{itemize}
In particular, for $1\leq i\leq k_2$ we have $y_i \in A+A+\{0, 1\}$, while for $k_2+3 \leq i\leq k_1$ we have $y_i\in A$. Let $i_0$ be the largest integer $\leq l$ such that $y_1, \ldots, y_{i_0} \in A+A+\{0, 1\}$. If $k_2+3>k_1$ then $i_0 = k_1+O(1) = k_2 + O(1)$. Otherwise, since $A$ and $A+A+\{0, 1\}$ are disjoint by \cref{lem:specialset}~\cref{item:disjoint}, we have $k_2 \leq i_0 < k_2+3$. Hence, $i_0 = k_2 + O(1)$ in both cases.

Repeating the whole argument with $f_3$ and $f_4$ in place of $f_1$ and $f_2$, we obtain that $i_0 = k_4 + O(1)$, whence $k_2 = k_4 + O(1)$ as claimed.

Our previous computations have shown that
\begin{equation*}
	x_i = \fmod{e_i(f_1) + e_i(f_2)}{q^{2i-1}-1} = \fmod{e_i(f_3) + e_i(f_4)}{q^{2i-1}-1}
\end{equation*}
for $1\leq i \leq \min(k_2, k_4)$, and 
\begin{equation*}
	x_i = e_i(f_1) = e_i(f_3)
\end{equation*}
for $\max(k_2, k_4)+3 \leq i\leq \min(k_1, k_3)$.

By definition of $e_i(f)$ and the Chinese remainder theorem, this implies that 
\begin{equation}
\label{eq:congruence2}
	f_1f_2 \equiv f_3f_4 \ \  \Big( \mathrm{mod}{\prod_{i=1}^{\min(k_2, k_4)} g_i}\Big)
\end{equation}
and
\begin{equation}
\label{eq:congruence3}
	f_1 \equiv f_3 \ \  \Big( \mathrm{mod}{\prod_{i = \max(k_2, k_4)+3}^{\min(k_1, k_3)} g_i}\Big).
\end{equation}

Suppose first that $f_1 = f_3$. Then $n_{f_2} = n - n_{f_1} = n - n_{f_3} = n_{f_4}$ and hence $f_2 = f_4$ by \cref{lem:preliminary}~\cref{item:uniqueness}, which is what we needed to prove. 

If $f_1f_2 = f_3f_4$, we immediately conclude that $\{f_1, f_2\} = \{f_3, f_4\}$, since all $f_i$ are irreducible monic polynomials.

Hence, we may suppose that $f_1 \neq f_3$ and $f_1f_2 \neq f_3f_4$. By \cref{eq:congruence3} we must have
\begin{equation*}
	\cst k_1^2+O(k_1) \geq \deg(f_1-f_3) \geq \deg \bigg( \prod_{i = k_2+O(1)}^{k_1-O(1)} g_i \bigg) = k_1^2-k_2^2 \optionalsign O(k_1).
\end{equation*}
Similarly, by \cref{eq:congruence2}, we have
\begin{equation*}
	 \cst k_1^2 + \cst k_2^2 + O(k_1) \geq \deg(f_1f_2 - f_3f_4) \geq \deg \bigg( \prod_{i=1}^{k_2 - O(1)} g_i \bigg) = k_2^2 - O(k_2).
\end{equation*}
Therefore, we have
\begin{equation*}
\left\{\begin{array}{rcl}
        k_1^2 - k_2^2 &\leq& \cst k_1^2 + O(k_1)\\
	k_2^2 &\leq& \cst k_1^2 + \cst k_2^2 + O(k_1).
\end{array}\right.
\end{equation*}
Hence, 
\begin{equation*}
	(1-\cst) k_1^2 - O(k_1) \leq k_2^2 \leq \frac{\cst}{1-\cst} k_1^2 + O(k_1),
\end{equation*}
which is impossible since $1-\cst > \frac{\cst}{1-\cst}$, if $\Cst$ is sufficiently large.
\end{proof}

\section{Asymptotic basis of order 3}

\begin{lemma}
\label{lem:expression}
Let $m \geq 3$ be an integer. We can write $m = \digits{z\, y_k\, x_k\, \cdots \, y_2\, x_2\, y_1\, x_1}$, for some $k\geq 0$ and some integers $x_i, y_i, z$ satisfying
\begin{itemize}
	\item $0\leq x_i < q^{2i-1}-1$ for all $1\leq i\leq k$,
	\item $0\leq y_i < 2\pp$ for all $1\leq i\leq k$,
	\item $\{y_i-2, y_i-1, y_i\} \subset A+A+A$ for all $1\leq i\leq k$,
	\item $3\leq z \leq 6\pp q^{2k+1}$.
\end{itemize} 
\end{lemma}
\begin{proof}
We proceed in a similar way to the standard algorithm that generates the base-$b$ expansion of an integer. Let us inductively define a sequence $(m_i)$ of integers $\geq 3$. Let $m_1 := m$. Suppose we have defined $m_1, \ldots, m_{2l-1}$ for some $l\geq 1$. 

If $m_{2l-1} > 6\pp q^{2l-1}$, we define 
\begin{equation*}
	x_l := \fmod{m_{2l-1}}{q^{2l-1}-1}
\end{equation*}
and we set $m_{2l} := ({m_{2l-1} - x_l})/({q^{2l-1}-1})$.

Let $0\leq y_l < 2\pp$ be an integer with ${y_l \equiv m_{2l} \pmod \pp}$ and such that ${\{y_l-2, y_l-1, y_l\} \subset A+A+A}$. Such an integer exists by \cref{lem:specialset}~\cref{item:covers} (note that $A+A+A \subset [0, 3p/2]$). We define $m_{2l+1} := (m_{2l} - y_l)/{\pp}$. Since $m_{2l-1} > 6\pp q^{2l-1}$, we have
\begin{equation*}
	m_{2l+1} \geq \frac{m_{2l}}{\pp} - 2 \geq \frac{m_{2l-1}}{\pp q^{2l-1}} - 3 \geq 3.
\end{equation*}
We have thus constructed $x_l$, $y_l$, $m_{2l}$ and $m_{2l+1}$.

Otherwise, if $m_{2l-1} \leq 6\pp q^{2l-1}$, so we can define $z := m_{2l-1}$ and stop the construction (thus $k=l-1$). 

The procedure must stop since $(m_i)$ is decreasing. When the procedure terminates, we end up with integers $x_1, \ldots, x_k$, $y_1, \ldots, y_k$ and $z$ that satisfy the required properties, by construction.
\end{proof}

\begin{lemma}
\label{lem:sawin}
Let $d\geq 1$ and let $g\in \poly$ be a squarefree monic polynomial of degree 
\begin{equation*}
	2\leq \deg(g) \leq 3\theta d
\end{equation*}
for some $0 < \theta < 1$. Let $a \in (\poly/(g))^{\times}$. Then, with $\phi(g) := \abs{\left(\poly/(g)\right)^{\times}}$, we have
\begin{equation*}
	\Bigg| \sum_{\substack{f_1, f_2, f_3 \in \irred{d}\\ f_i~\mathrm{distinct}\\ f_1f_2f_3 \equiv a \ (\mathrm{mod}(g))}} \hspace{-0.3cm} 1 \ -\  \frac{1}{\phi(g)} { \abs{\irred{d}}\choose 3} \Bigg| \ll e^{O_{\theta}(d)} q^{\frac{3d - \deg(g)}{2}}.
\end{equation*}
\end{lemma}

\begin{proof}
We use \cite[Lemma~9.14]{sawin} with $\omega = 3$, $n_1 = n_2 = n_3 = d$, $n = 3d$ and $h = 1$ (meaning that $c=0$). The author writes $\poly^{+}$ and $\mathcal{M}_n$ for the set of monic polynomials in $\poly$ and the set of monic polynomials of degree $n$ in $\poly$, respectively. The definition of $H^{h}_{n_1, \ldots, n_{\omega}}(f)$ can be found at the bottom of \cite[p.89]{sawin}. In particular,
\begin{equation*}
	H^{1}_{d, d, d}(f) = d^3 \sum_{\substack{f_1, f_2, f_3 \in \irred{d}\\ f_i \text{ distinct}\\ f_1f_2f_3 = f}}  1,
\end{equation*}
and thus 
\begin{equation*}
	\sum_{f\in \mathcal{M}_{3d}}H^{1}_{d, d, d}(f) = d^3 { \abs{\irred{d}}\choose 3}.
\end{equation*}
Since $c = 0$, the conclusion of \cite[Lemma~9.14]{sawin} simplifies to 
\begin{equation*}
	d^3 \Bigg| \sum_{\substack{f_1, f_2, f_3 \in \irred{d}\\ f_i \text{ distinct}\\ f_1f_2f_3 \equiv a \ (\mathrm{mod}(g))}} \hspace{-0.3cm} 1 \ -\  \frac{1}{\phi(g)} { \abs{\irred{d}}\choose 3} \Bigg| \ll \frac{(3d)!}{(d!)^3} (O_{\theta} (q))^{\frac{3d - \deg(g)}{2}}.
\end{equation*}
By Stirling's formula, $\frac{(3d)!}{(d!)^3} \ll e^{O(d)}$, and we obtain \cref{lem:sawin}.
\end{proof}

\begin{lemma}
\label{lem:basis3individual}
Let $m\geq q^{(C_0+2)^2}$. Then $\Pr{m\not\in S+S+S} \leq \exp(-m^{3c-1 \optionalsign o(1)})$ as $m\to +\infty$.\footnote{Recall that $\Prbis$ refers to the probability in the random choice of $r_i(f), s(f)$ (which are the random variables used to define~$S$).}
\end{lemma}
\begin{proof}
Let $m \geq q^{(C_0+2)^2}$. We write $m = \digits{z\, y_k\, x_k\, \cdots \, y_2\, x_2\, y_1\, x_1}$ with $x_i$, $y_i$ and $z$ satisfying the conclusion of \cref{lem:expression}. By a computation similar to that in \cref{lem:preliminary}~\cref{item:sizenf}, we have 
\begin{equation*}
	q^{k^2} \leq m \leq q^{(k+2)^2}.
\end{equation*}
In particular, $k \geq C_0$ and $m = q^{k^2+O(k)}$.

We will see that $m$ is very likely to be expressible as $n_{f_1}+n_{f_2}+n_{f_3}$ for some $f_i\in \F_k$. 

We first focus on the digits $x_i$. By the Chinese remainder theorem, there is some $a\in \poly$ such that 
\begin{equation*}
	a \equiv \omega_i^{x_i} \pmod{g_i}
\end{equation*}
for all $1\leq i\leq k$. Let $g = \prod_{1\leq i\leq k} g_i$; it is a squarefree polynomial of degree $\deg(g) = k^2$. Let $d$ be an even integer with $\cst k^2 \leq d < \cst(k+1)^2$. Let $\E$ be the set of all triples $(f_1, f_2, f_3) \in (\irred{d})^3$ of distinct polynomials such that $f_1f_2f_3 \equiv a \pmod{g}$. By \cref{lem:sawin} with $\theta = k^2/(3d) \leq 1/(3c) < 1$, we have 
\begin{equation}
\label{eq:firstboundE}
	\abs{\abs{\E} - \frac{1}{\phi(g)} { \abs{\irred{d}}\choose 3}} \ll e^{O(d)} q^{\frac{3d-k^2}{2}},
\end{equation}
Using the fact that $\phi(g) = \prod_{i=1}^k (q^{2i-1}-1) = q^{k^2-O(k)}$ and the bound \cref{eq:gauss} on the size of $\irred{d}$, we can simplify \cref{eq:firstboundE} to get
\begin{equation*}
	\abs{\E} = q^{3d-k^2+O(k+\log d)} + O\left(q^{\frac{3d-k^2}{2}+O\big(\frac{d}{\log q}\big)}\right) = q^{(3\cst - 1)k^2+O(k)} + O\left(q^{\big(\frac{3\cst -1}{2}+O\big(\frac{1}{\log q}\big)\big)k^2}\right),
\end{equation*}
using $d = \cst k^2 + O(k)$ for the last step. If $\Cst$ is sufficiently large, the error term is negligible and we obtain
\begin{equation}
\label{eq:thirdboundE}
	\abs{\E} = q^{(3\cst - 1)k^2+O(k)}.
\end{equation}

Let us estimate, for a fixed triple $(f_1, f_2, f_3)\in \E$, the probability that $n_{f_1}+n_{f_2}+n_{f_3} = m$. By definition of the $n_{f_i}$, we can write 
\begin{equation*}
	n_{f_1}+n_{f_2}+n_{f_3} = \digits{s'\, r_k'\, e_k'\, \ldots\, r_2'\, e_2'\, r_1'\, e_1'}
\end{equation*}
where $e_i'$, $r_i'$ and $s'$ are defined by
\begin{itemize}
	\item $e_i' = \fmod{e_i(f_1) + e_i(f_2) + e_i(f_3)}{q^{2i-1}-1}$ for $1\leq i\leq k$,
	\item $r_i' = r_i(f_1) + r_i(f_2) + r_i(f_3) + \kappa_i$ for $1\leq i\leq k$, with $\kappa_i := \lfloor \frac{e_i(f_1) + e_i(f_2) + e_i(f_3)}{q^{2i-1}-1} \rfloor \in \{0, 1, 2\}$,
	\item $s' = s(f_1) + s(f_2) + s(f_3)$.
\end{itemize}
The assumption that $(f_1, f_2, f_3)\in \E$ ensures that $e_i' = x_i$ for all $1\leq i \leq k$. Indeed, both $e_i'$ and $x_i$ are between $0$ and $q^{2i-1}-2$, and 
\begin{equation*}
	\omega_i^{e_i'} \equiv \omega_i^{e_i(f_1) + e_i(f_2) + e_i(f_3)} \equiv f_1f_2f_3 \equiv a \equiv \omega_i^{x_i} \pmod{g_i}.
\end{equation*}
Hence,
\begin{equation*}
	\Pr{n_{f_1}+n_{f_2}+n_{f_3} = m} \geq \Pr{r_1' = y_1,\,r_2' = y_2,\,\ldots,r_k' = y_k,\, s' = z} = \Pr{s' = z} \prod_{1\leq i\leq k}\Pr{r_i' = y_i}.
\end{equation*}
where we used the independence of the family of all random variables $r_i(f_j)$ and $s(f_j)$ in the last step.

Recall that $s'$ is the sum of three independent random numbers chosen uniformly in $\{1, 2, \ldots, q^{3k}\}$, so $s'$ can be equal to any integer in the set $\{3, 4, \ldots, 3q^{3k}\}$, each of which occurs with probability at least $q^{-9k}$. Since $z$ is a fixed element of $\{3, 4, \ldots, 6\pp q^{2k+1}\}$ and $6\pp q^{2k+1} \leq 3q^{3k}$, we get 
\begin{equation*}
	\Pr{s' = z} \geq q^{-9k}.
\end{equation*}
Similarly, $r_i'$ is the sum of three elements of $A$ chosen uniformly and independently at random, plus a fixed carry $\kappa_i \in \{0, 1, 2\}$. Since $\{y_i-2, y_i-1, y_i\} \subset A+A+A$ by \cref{lem:expression}, there is at least one triple $(a_1, a_2, a_3)\in A^3$ such that $a_1+a_2+a_3+\kappa_i = y_i$. Thus, we see that $\Pr{r_i' = y_i} \geq |A|^{-3} \geq \pp^{-3}$. We have thus shown that 
\begin{equation*}
	\Pr{n_{f_1}+n_{f_2}+n_{f_3} = m} \geq q^{-9k}\pp^{-3k} \geq q^{-O(k)}.
\end{equation*}

To conclude the argument, we wish to restrict to a large subset $\E_0$ of $\E$ such that no polynomial appears in more than one triple of $\E_0$. Observe that, if two triples $(f, f_1, f_2)$ and $(f, f_3, f_4)$ are both in $\E$, then $\{f_1, f_2\} = \{f_3, f_4\}$. Indeed, two such triples satisfy $f_1f_2 \equiv f^{-1}a \equiv f_3f_4 \pmod{g}$, which implies that $f_1f_2 = f_3f_4$ in $\poly$ since 
\begin{equation*}
	\deg(f_1f_2 - f_3f_4) \leq 2d < k^2 = \deg(g),
\end{equation*}
and thus $\{f_1, f_2\} = \{f_3, f_4\}$. In addition, if $(f_1, f_2, f_3)\in \E$ then so does $(f_{\sigma(1)}, f_{\sigma(2)}, f_{\sigma(3)})$ for every permutation $\sigma$. Therefore, any given polynomial $f\in \irred{d}$ appears at most once in a triple of $\E$, up to permutations. This shows that there is a suitable subset $\E_0$ with $\abs{\E_0} = \tfrac16 \abs{\E}$.

Therefore, 
\begin{align*}
\Pr{m\not\in S+S+S} &\leq \Prbis \Bigg(\bigcap_{(f_1, f_2, f_3)\in \E_0} \{n_{f_1}+n_{f_2}+n_{f_3} \neq m\} \Bigg)\\
&= \prod_{(f_1, f_2, f_3)\in \E_0} \Pr{n_{f_1}+n_{f_2}+n_{f_3} \neq m} \leq \big(1-q^{-O(k)}\big)^{|\E_0|},
\end{align*}
using the independence of the family of random variables $(n_{f_1}+n_{f_2}+n_{f_3})_{(f_1, f_2, f_3)\in \E_0}$. By \cref{eq:thirdboundE} and the simple inequality $1-x\leq \exp(-x)$, we get
\begin{equation*}
	\Pr{m\not\in S+S+S} \leq \exp \left(-q^{-O(k)} q^{(3\cst - 1)k^2-O(k)} \right) \leq \exp \left(- m^{(3c-1+o(1))} \right),
\end{equation*}
as $m = q^{k^2+O(k)}$.
\end{proof}

\begin{theorem}
\label{thm:basis3}
With probability $1$, the elements of $S$ form a Sidon sequence and an asymptotic basis of order~$3$.
\end{theorem}
\begin{proof}
By \cref{lem:sidon}, $S$ always has the Sidon property. By \cref{lem:basis3individual}, 
\begin{equation*}
	\Pr{m\not\in S+S+S} \leq \exp(-m^{3c-1 \optionalsign o(1)}) \ll m^{-2}
\end{equation*}
for sufficiently large $m$, so 
\begin{equation*}
	\Pr{\abs{\N \setminus S+S+S} = +\infty} = 0
\end{equation*}
by the Borel-Cantelli lemma. We conclude that $S$ is almost surely an asymptotic basis of order~$3$.
\end{proof}

\appendix

\section{Auxiliary set}
\label{sec:appendix}

In this section, we use the notation $f*g$ for the additive convolution
\begin{equation*}
	f*g (n) = (f*g) (n) := \sum_{\substack{a, b\in \Z\\ a+b=n}} f(a)g(b).
\end{equation*}

\begin{proof}[Proof of \cref{lem:specialset}]
Let $I = \{1, 2, \ldots, \lfloor \pp/2 \rfloor - 1\} \subset \Z$. 

Let $R$ be the random set obtained by selecting each element of $I$, independently, with probability $Kp^{-2/3}$, where $K = \lceil \log p \rceil$. In particular, $\EE{\abs{R}} \asymp Kp^{1/3}$.

Note that $R$ almost satisfies property \cref{item:disjoint} of \cref{lem:specialset}, in the sense that
\begin{equation*}
	\EEbis \big[\abs{R\cap (R+R+\{0, 1\})} \big] \leq \sum_{\substack{a,b,c \in I\\ a-b-c \in \{0, 1\}}} \Pr{a,b,c\in R}.
\end{equation*}
The probability $\Pr{a,b,c\in R}$ is $(Kp^{-2/3})^k$ where $k = |\{a,b,c\}|$. Therefore,
\begin{equation}
\label{eq:propjustfails}
	\EEbis \big[\abs{R\cap (R+R+\{0, 1\})} \big] \ll p^2 (Kp^{-2/3})^3 + p (Kp^{-2/3})^2 + 1 (Kp^{-2/3})^1 \ll K^3.
\end{equation}

It will be useful to know the bound $\norm{\ind{R} *\ind{R}}_{\infty} \ll 1$, which holds with high probability. Indeed, if $\ind{R}*\ind{R}(n) \geq 9$ for some $n\in \N$, then $R$ contains elements $a_1<a_2<\cdots<a_8$ such that $a_i+a_{9-i} = n$ for all $i$, and thus
\begin{equation}
\label{eq:infinitenorm}
	\Pr{\exists n, \  \ind{R}*\ind{R}(n) \geq 9} \leq \sum_{n=1}^{p} \sum_{\substack{a_1, \ldots, a_8\in I\\ a_i\text{ distinct}\\ a_i+a_{9-i} = n}} \Pr{\forall i,\, a_i\in R} \ll p^5 (Kp^{-2/3})^8 \ll K^8p^{-1/3} \ll K^{-1}.
\end{equation}

Similarly, if $\ind{R}*\ind{-R}(n) \geq 8$ for some $n\in \Z^{\neq 0}$, then $R$ contains distinct elements $a_1,\ldots, a_8$ such that $a_{2i}-a_{2i-1} = {n}$ for $i\in \{1,2,3,4\}$, and as above we get
\begin{equation}
\label{eq:infinitenormdiff}
	\Pr{\exists n\in \Z^{\neq 0}, \  \ind{R}*\ind{-R}(n) \geq 8} \ll K^{-1}.
\end{equation}

We now turn to property \cref{item:covers} of \cref{lem:specialset}. Let $n$ be an integer with $p/5 \leq n \leq 7p/5$. We will show that $n$ can be written in relatively many ways as the sum of three distinct elements of $R$, with high probability. Let $\triples(n)$ be the collection of all sets $\{a, b, c\}$ of three pairwise distinct elements of $I$ such that $a+b+c = n$. We will show that 
\begin{equation}
\label{eq:toprovebyJanson}
	\Prbis\Bigg( \sum_{\substack{T\in \triples(n)}} \ind{T\subset R} \leq K^{2}\Bigg) \ll e^{-K^2}.
\end{equation}
To bound the left-hand side, we use Janson's inequality. This inequality involves the quantities
\begin{equation*}
	\lambda := \EEbis \Bigg[ \sum_{\substack{T\in \triples(n)}} \ind{T\subset R} \Bigg] = \sum_{\substack{T\in \triples(n)}} (Kp^{-2/3})^3  \gg p^2(Kp^{-2/3})^3 = K^3
\end{equation*}
and
\begin{equation*}
	\Delta := \sum_{\substack{T_1, T_2\in \triples(n)\\ T_1\cap T_2\neq \emptyset \\ T_1\neq T_2}} \Pr{T_1\cup T_2\subset R}. 
\end{equation*}
Janson's inequality \cite[Theorem~2.14]{janson} gives
\begin{equation}
\label{eq:jansonused}
	\Prbis\Bigg( \sum_{\substack{T\in \triples(n)}} \ind{T\subset R} \leq \tfrac12 \lambda\Bigg) \leq \exp \left( -\frac{\lambda^2}{8(\lambda+\Delta)}\right).
\end{equation}

Observe that, if $T_1, T_2\in \triples(n)$ are distinct but not disjoint, we must have $\abs{T_1 \cap T_2} = 1$, so there are pairwise distinct elements $a, b_1, b_2, c_1, c_2\in I$ such that $T_1 = \{a, b_1, c_1\}$ and $T_2 = \{a, b_2, c_2\}$. In particular,
\begin{equation*}
	\Pr{T_1\cup T_2\subset R} = (K p^{-2/3})^5,
\end{equation*}
and there are $\ll p^3$ such pairs $(T_1, T_2)$. Hence,
\begin{equation*}
	\Delta \ll p^3 (Kp^{-2/3})^5 \ll 1 \ll \lambda.
\end{equation*}
Provided $p$ is larger than some absolute constant, we have $K^{2} \leq \tfrac12 \lambda$, so \cref{eq:toprovebyJanson} follows from \cref{eq:jansonused}, using our estimates for $\lambda$ and $\Delta$. By the union bound, we get
\begin{equation}
\label{eq:unionbound}
	\Prbis \Bigg( \exists n\in [p/5, 7p/5]\cap \Z,\  \sum_{\substack{T\in \triples(n)}} \ind{T\subset R} \leq K^{2} \Bigg) \ll p e^{-K^2} \ll K^{-1}.
\end{equation}

We have shown that, typically, $R$ satisfies property \cref{item:covers} of \cref{lem:specialset} in a robust sense (by \cref{eq:unionbound}), but $R$ just fails to satisfy \cref{item:disjoint} (see \cref{eq:propjustfails}). 

We define $X := R\cap (R+R+\{0, 1\})$ and $A := R\setminus X$. This set $A$ satisfies property \cref{item:disjoint} of \cref{lem:specialset} by construction. We must show that $A$ still satisfies property \cref{item:covers} of \cref{lem:specialset}, with high probability. 

Let $n\in [p/5, 7p/5]\cap \Z$. We know that, with high probability, $n$ is the sum of three distinct elements of $R$ in many different ways. Thus, having $n\not\in A+A+A$ means that, whenever $n$ is written as $n = a+b+c$ with distinct $a,b,c\in R$, at least one of $a,b,c$ is in $X$.

Define
\begin{equation*}
	\tupples(n) := \Big\{(a,b,c,d,e) \in I^5: a,b,c \text{ pairwise distinct}, \ a+b+c = n, \ c-d-e\in \{0, 1\}\Big\}.
\end{equation*}
Suppose that $n$ is such that
\begin{equation*}
	\abs{\tupples(n) \cap R^5}\geq K.
\end{equation*}
We claim that either $\max_{m\in \N} \ind{R} *\ind{R}(m) \geq 10$, $\max_{m\in \N} \ind{R} *\ind{-R}(m) \geq 10$, or there are four tuples $(a_i, b_i, c_i, d_i, e_i)_{1\leq i \leq 4} \in \tupples(n) \cap R^5$ such that, for all $1\leq i<j\leq 4$,
\begin{equation}
\label{eq:conddisjointcoord}
	\{a_i, b_i, c_i, d_i, e_i\} \cap \{a_j, b_j, c_j, d_j, e_j\} = \emptyset.
\end{equation}

Suppose first that there is some $a\in R$ appearing as the first coordinate of $200$ tuples 
\begin{equation*}
	{(a, b_i, c_i, d_i, e_i)_{1\leq i\leq 200} \in \tupples(n) \cap R^5}.
\end{equation*}
Then $n-a$ can be written as $b_i+c_i$ for all $1\leq i\leq 200$. If $\ind{R} *\ind{R}(n-a) < 10$, there are $b,c\in R$ such that $b_i = b$ and $c_i = c$ for at least $20$ values of $i\in \{1, \ldots, 200\}$, say for $i \in J$ where $|J| \geq 20$. In turn, this implies that $d_i + e_i \in \{c-1, c\}$ for all $i\in J$, but if $\ind{R} *\ind{R}(c-1) < 10$ and $\ind{R} *\ind{R}(c) < 10$ there must exist two distinct indices $i, j\in J$ such that $(d_i, e_i) = (d_j, e_j)$. This is impossible since $(a, b_i, c_i, d_i, e_i)$ and $(a, b_j, c_j, d_j, e_j)$ are distinct tuples.

Similar reasoning shows that no $r\in R$ can appear as the second or third coordinate of $200$ tuples in $\tupples(n) \cap R^5$, unless $\max_{m\in \N} \ind{R} *\ind{R}(m) \geq 10$. 

Suppose that there is some $d\in R$ that appears as the fourth coordinate of $200$ tuples 
\begin{equation*}
	{(a_i, b_i, c_i, d, e_i)_{1\leq i\leq 200} \in \tupples(n) \cap R^5}.
\end{equation*}
Then $c_i-e_i \in \{d, d+1\}$ for all $i$. If $\max_{m\in \N} \ind{R} *\ind{-R}(m) < 10$, there are $c, e\in R$ and at least $10$ values of $i$ such that $c_i = c$ and $e_i = e$. Thus, for these $\geq 10$ values of $i$, we have $a_i + b_i = n-c$. If $ \ind{R} *\ind{R}(n-c) < 10$, there are two distinct indices $i,j$ such that $(a_i, b_i, c_i, d, e_i) = (a_j, b_j, c_j, d, e_j)$, a contradiction.

The same reasoning shows that no $e\in R$ appears in $200$ distinct elements of $\tupples(n) \cap R^5$, unless $\norm{\ind{R} *\ind{R}}_{\infty} \geq 10$ or $\max_{m\in \N} \ind{R} *\ind{-R}(m) \geq 10$.

The above claim follows easily from these observations. Indeed, suppose that $\norm{\ind{R} *\ind{R}}_{\infty} < 10$, $\max_{m\in \N} \ind{R} *\ind{-R}(m) < 10$ and $\abs{\tupples(n) \cap R^5}\geq K$. Let $(a_1, b_1, c_1, d_1, e_1)$ be an arbitrary element of $\tupples(n) \cap R^5$. The previous observations show that the number of tuples in $\tupples(n) \cap R^5$ having a coordinate in $\{a_1, \ldots, e_1\}$ is bounded above by an absolute constant. Since $K$ is assumed to be sufficiently large, we can find another tuple $(a_2, b_2, c_2, d_2, e_2)$ such that $\{a_1, \ldots, e_1\} \cap \{a_2, \ldots, e_2\} = \emptyset$. Repeating, we find four tuples $(a_i, b_i, c_i, d_i, e_i)_{1\leq i\leq 4}\in \tupples(n) \cap R^5$ satisfying \cref{eq:conddisjointcoord}.

By this claim and the union bound, we deduce that
\begin{align*}
	\Prbis \Big(\exists n\in [p/5, 7p/5]\cap \Z,\ &\abs{\tupples(n) \cap R^5}\geq K\Big) \leq \Prbis\Big( \norm{\ind{R} *\ind{R}}_{\infty} \geq 10\Big) + \Prbis \Big(\max_{m\in \N} \ind{R} *\ind{-R}(m) \geq 10\Big) \\&+ \sum_{n\in [p/5, 7p/5]\cap \Z} \ \sum_{\substack{(a_i, b_i, c_i, d_i, e_i)_{1\leq i\leq 4} \in \tupples(n) \\ \text{satisfying \cref{eq:conddisjointcoord}}}} \Prbis \Big(\forall i,\  a_i, b_i, c_i, d_i, e_i\in R \Big).
\end{align*}
By \cref{eq:infinitenorm} and \cref{eq:infinitenormdiff}, the probabilities $\Prbis( \norm{\ind{R} *\ind{R}}_{\infty} \geq 10)$ and $\Prbis (\max_{m\in \N} \ind{R} *\ind{-R}(m) \geq 10)$ are $\ll K^{-1}$. By \cref{eq:conddisjointcoord}, the events $\{a_i, \ldots, e_i\in R\}$ and $\{a_j, \ldots, e_j\in R\}$ are independent for $i\neq j$. Thus, the last probability is 
\begin{equation*}
	\Prbis (\forall i,\  a_i, b_i, c_i, d_i, e_i\in R) = \prod_{i=1}^4 \Pr{a_i, b_i, c_i, d_i, e_i\in R} = \prod_{i=1}^4 (Kp^{-2/3})^{\abs{\{a_i, \ldots, e_i\}}}.
\end{equation*}
For $k\in \{3,4,5\}$, let  
\begin{equation*}
	\tupples_k(n) = \Big\{ (a,b,c,d,e)\in \tupples(n) : \abs{ \{a,b,c,d,e\}} = k \Big\}.
\end{equation*}
It is not hard to see that $\abs{\tupples_k(n)} \ll p^{k-2}$ for $k\in \{3,4,5\}$. Hence, we obtain
\begin{align*}
	\Prbis \Big(\exists n\in [p/5, 7p/5]\cap \Z,\ \abs{\tupples(n) \cap R^5}\geq K\Big) &\ll K^{-1} + \sum_{n\in [p/5, 7p/5]\cap \Z}\, \prod_{i=1}^4  \sum_{k =3}^5 \abs{\tupples_{k}(n)} (Kp^{-2/3})^{k}\\
	&\ll K^{-1} + p \bigg( \sum_{k=3}^5 p^{k-2} (Kp^{-2/3})^k \bigg)^4 \\
	&\ll K^{-1} + p (K^{5} p^{-1/3})^4 \\ &\ll K^{-1}. 
\end{align*}

The last computation and \cref{eq:unionbound} imply that there is a set $R\subset I$ such that 
\begin{enumerate}[label=(\alph*), ref=\alph*]
	\item for all $n\in [p/5, 7p/5]\cap \Z$, there are $>K^2$ sets $\{a,b,c\}$ of three distinct elements of $R$ such that $n = a+b+c$;
	\item for all $n\in [p/5, 7p/5]\cap \Z$, there are $\ll K$ sets $\{a,b,c\}$ of three distinct elements of $R$ such that $n = a+b+c$, with one of $a,b,c$ in $R+R+\{0, 1\}$,
\end{enumerate}
since with high probability, both properties hold simultaneously. 

As announced earlier, we define $X = R\cap (R+R+\{0, 1\}$ and $A = R\setminus X$. Then $A$ and $A+A+\{0, 1\}$ are disjoint by construction, and for every $n\in [p/5, 7p/5]\cap \Z$ there are $K^2 - O(K) \geq 1$ ways to write $n$ as the sum of three elements of $R$, none of which being in $X$. This means that we can write $n$ as the sum of three elements of $A$, and we are done.
\end{proof}

\bibliography{proof}
\bibliographystyle{amsplain}

\end{document}